\documentclass[pdflatex,sn-mathphys-num,colorlinks=true, linkcolor=blue, urlcolor=blue, citecolor=blue]{sn-jnl} 
\usepackage{bm}
\usepackage{graphicx}
\usepackage{adjustbox}
\usepackage{amsthm,amssymb,amsmath}
\usepackage[T1]{fontenc}
\usepackage[utf8]{inputenc}
\usepackage{enumitem}
\usepackage{array}
\usepackage{setspace}
\usepackage{float}
\usepackage{mathtools}
\usepackage{extarrows}
\usepackage[all,cmtip]{xy}
\usepackage{tikz-cd}
\usepackage{relsize}
\usepackage{pifont}

\newtheorem{lemma}{Lemma}[section]

\newtheorem{theorem}[lemma]{Theorem}

\newtheorem{proposition}[lemma]{Proposition}
\newtheorem{corollary}[lemma]{Corollary}

\newtheorem{definition}[lemma]{Definition}
\newtheorem{remark}[lemma]{Remark}

\usepackage{blindtext}

\usepackage{graphicx}%
\usepackage{multirow}%

\usepackage{amsthm}%

\usepackage[title]{appendix}%
\usepackage{xcolor}%
\usepackage{textcomp}%
\usepackage{manyfoot}%
\usepackage{booktabs}%
\usepackage{algorithm}%
\usepackage{algorithmicx}%
\usepackage{algpseudocode}%
\usepackage{listings}%
\numberwithin{equation}{section}

\newcommand{\R}{\mathbb{R}}

\begin{document}

\title[Positive cones of a finite quotient]{On positive cones of finite quotients of a normal variety}

\author[1]{\fnm{Ashima} \sur{Bansal}}
\email{ashima.bansal@snu.edu.in}

\author[2]{\fnm{Indranil} \sur{Biswas}}
\email{indranil.biswas@snu.edu.in, indranil29@gmail.com}

\author[3]{\fnm{Souradeep} \sur{Majumder}}
\email{souradeep@labs.iisertirupati.ac.in}

\affil[1]{\orgdiv{Department of mathematics}, \orgname{Shiv Nadar University}, \orgaddress{\street{NH91, Tehsil Dadri}, \city{Greater Noida}, \postcode{201314}, \state{Uttar Pradesh}, \country{India}}}

\affil[2]{\orgdiv{Department of mathematics}, \orgname{Shiv Nadar University}, \orgaddress{\street{NH91, Tehsil Dadri}, \city{Greater Noida}, \postcode{201314}, \state{Uttar Pradesh}, \country{India}}}

\affil[3]{\orgdiv{Department of mathematics}, \orgname{Indian Institute of Science Education and Research Tirupati}, \orgaddress{\street{Srinivasapuram, Yerpedu Mandal}, \city{Tirupati}, \postcode{517619}, \state{Andhra Pradesh}, \country{India}}}

\abstract{We study positivity properties of finite flat quotients of normal projective varieties. We show that the numerical groups and positive cones of these quotient varieties are related to those of the original varieties.


}

\footnotetext[1]{ \texttt{ashima.bansal@snu.edu.in}}
\footnotetext[2]{ \texttt{indranil.biswas@snu.edu.in, indranil29@gmail.com}}
\footnotetext[3]{ \texttt{souradeep@labs.iisertirupati.ac.in}}

\maketitle

\section{Introduction}\label{sec1}

Let $X$ be a normal projective variety over the complex numbers. Denote by $N_{k}(X)$ the 
numerical group of dimension $k$-cycles with $\mathbb{R}$-coefficients. The pseudoeffective cone 
$\overline{\textrm{Eff}}_{k}(X)$ is defined as the closure in $N_{k}(X)$ of the cone generated by classes 
of $ k$-dimensional subvarieties of $X$. When $X$ is smooth, the nef cone $\textrm{Nef}^{\,k}(X)$ is defined as the dual of $\overline{\textrm{Eff}}_{k}(X)$
with respect to the intersection pairing. When $X$ is singular, we instead work with the space of dual numerical classes $N^{k}(X)$, which is the abstract dual of $N_{k}(X)$.

It is well known that proper pushforward maps between projective varieties descend naturally to numerical 
groups. Fulger
and Lehmann \cite{FL} showed that, for a dominant morphism between projective varieties,
the induced pushforward maps the pseudoeffective cone surjectively
onto the corresponding pseudoeffective cone of the target. Regarding pullbacks, Dang \cite{Da} proved that the flat pullback for flat morphisms between normal projective varieties descends to numerical groups.

Quotients by group actions play an important role in algebraic geometry, particularly in the classification of geometric objects. Many naturally occurring moduli and classifying spaces arise by identifying points under group actions. Geometric Invariant Theory (GIT) provides the fundamental framework for constructing well-behaved quotients of varieties or schemes under the action of algebraic groups. Among the most classical and fundamental instances of such constructions are quotients of quasi-projective varieties by finite group actions, which serve as an important source of examples and motivate several questions concerning the geometry and positivity properties of the resulting quotient varieties.

In this article, we study the numerical groups and pseudoeffective cones of finite group quotients. Let $G$ be a finite group acting on a normal projective variety $X$, and let $\pi\,:\, X\,\longrightarrow\, X/G$ be the quotient map. In Section $3$, we work under the assumption that $\pi$ is flat. This ensures that the 
pullback maps between the numerical groups are defined, as showed by \cite{Da}. We then show that the special pullback map $\pi_{\textrm{sp}}^{\ast}\,:\, Z_{\ast}(X/G)\,\longrightarrow\,Z_{\ast}(X)$, defined 
in \cite[Example 1.7.6]{Fu}, agrees with the flat pullback. 

Our main results etablish natural isomorphisms between the numerical groups of $X/G$ and the $G$-invariant numerical groups of $X$; see Theorem \ref{lower} and Corollary \ref{invariant}. These isomorphisms identify the pseudoeffective cone 
$\overline{\textrm{Eff}}_{k}$, the  pseudoeffective dual cone
$\overline{\textrm{Eff}}^{\,k}$,  and the nef cone 
$\textrm{Nef}^{\,k}$ of $X/G$ with the corresponding $G$-invariant cones on $X$; 
see Theorems \ref{pullback}, \ref{pseu}, and Proposition \ref{nef}.

As an application, we show that the equality between
$\textrm{Nef}^{\,k}(X)^{G}\,=\, \overline{\textrm{Eff}}^{\,k}(X)^{G}$ holds if and only if the corresponding 
equality $\textrm{Nef}^{\,k}(X/G)\,=\, \overline{\textrm{Eff}}^{\,k}(X/G)$ holds for the quotient; see
Proposition \ref{equality}.

\section{Preliminaries}

This section recalls the numerical group of dimension $k$-cycles and its closed subcones (see \cite{FL} 
for more details).

\subsection{Chow groups and Numerical groups}

Throughout this subsection, $X$ is a complex projective variety of dimension $n$.

\subsubsection{Cycles and Chow groups}

Let $Z_{k}(X)$ denote the group of $k$--cycles on $X$ with coefficients in $\mathbb R$. Any subscheme $Z\,
\subset\, 
X$ of dimension $k$ has a fundamental cycle denoted by $[Z]\,\in\,{Z}_{k}(X)$ defined as in \cite[\S~1.5]{Fu}. 
To study the geometry of cycles on $X$, various equivalence relations have been 
introduced on $Z_{k}(X)$. One example is rational equivalence; see \cite[\S~1.3 and \S~1.6]{Fu}.
The Chow group ${\rm CH}_{k}(X)$ is the quotient of ${Z}_{k}(X)$ modulo rational equivalence, which may
still have infinite rank.

Let $E$ be a vector bundle on $X$. For  nonnegative integers $k$ and $m$, the $k$--th Chern class of $E$ defines a linear map
$${\rm CH}_{m}(X)\ \, \xrightarrow{\,\,\,c_{k}(E)\,\cap_{-}\,\,\,}\ \,{\rm CH}_{m-k}(X)$$ \cite[\S~3.2]{Fu}.
For notational convenience, we denote $c_{k}(E)\,\bigcap\, [X]$ simply by  $c_{k}(E)$.  Furthermore, there is a natural homomorphism
\[
\textrm{Pic}(X)
\,\longrightarrow\,
{\rm CH}_{n-1}(X)_{\mathbb Z},
\]
defined by
\[
\mathcal L\,\longmapsto\, c_{1}(\mathcal L)\,\cap\, [X].
\]
This homomorphism is injective when \(X\) is normal, and it is an
isomorphism when \(X\) is locally factorial, that is, when all the
local rings of \(X\) are unique factorization domains.

\medskip
\noindent\textbf{Smooth case.}
When \(X\) is smooth, we set
\[
{\rm CH}^{k}(X)\,:=\,{\rm CH}_{n-k}(X).
\]
The intersection product equips
\[
{\rm CH}^{\ast}(X)
\,:=\,
\bigoplus_{k\,\geq\, 0}{\rm CH}^{k}(X)
\]
with the structure of a graded ring; see \cite[\S~8]{Fu}.

\subsubsection{Chow groups of a finite group quotient}

Let $G$ be a finite group acting on a complex projective variety $X$, and let $Y\,=\,X/G$ be the quotient variety. Denote by $\pi\,:\,X\,\longrightarrow\,Y$ the finite quotient map.
Let ${\rm CH}_{\ast}(X)^{G}$ denote the  $G$-invariant subgroup of ${\rm CH}_{\ast}(X)$. Then there is a canonical isomorphism of groups ${\rm CH}_{\ast}(Y)\, =\, {\rm CH}_{\ast}(X)^{G}$.
For any subvariety $W$ of $X$, let 
\begin{equation*}
I_{W}\ :=\ \{g\,\in\,G\,\,\big\vert\,\, g_{|_W}\,=\, id_{W}\}    
\end{equation*}
be the inertia group, and set 
\begin{equation*}
e_{W}\,:=\, \textrm{card}(I_{W})/\textrm{deg}_{i}(W/V),    
\end{equation*}
where $V\,=\, \pi(W)$, and ${\rm deg}_{i}(W/V)$ is the degree of inseparability of $K(W)$ over $K(V)$, the field of rational functions on the varieties $W$ and $V$, respectively. We recall from \cite[Example 1.7.6]{Fu} that, for a subvariety $V$ of $Y$, one defines
\begin{equation}\label{special}
\pi_{\textrm{sp}}^{\ast}[V]\,=\, \sum\, e_{W} [W],    
\end{equation}
where the sum is taken over all irreducible components $W$ of $\pi^{-1}(V)$. This determines an isomorphism $Z_{\ast}(Y)\,=\, Z_{\ast}(X)^{G}$, and ${\rm CH}_{\ast}(X)^{G}$ is the quotient of $Z_{\ast}(X)^{G}$ modulo the subspace generated by
\begin{equation*}
\left\{\sum_{g\,\in\,G}g_{\ast}[\textrm{div}(r)]\,\,\big\vert\,\,r\,
\in\,K(W)^{\ast},\,\,\, W\,\subset\,X\right\}.    
\end{equation*}
Note that the composition 

\begin{equation}\label{comp}
{\rm CH}_{\ast}(Y)\,\xlongrightarrow{\,\,\pi_{\textrm{sp}}^{\ast}\,\,}\, {\rm CH}_{\ast}(X)^{G}\,
\hookrightarrow\,{\rm CH}_{\ast}(X)\,\xlongrightarrow{\,\,\pi_{\ast}\,\,}\, {\rm CH}_{\ast}(Y)     
\end{equation}
is multiplication by card($G$).

\medskip
\noindent\textbf{Smooth case.} If $X$ is a smooth projective variety then ${\rm CH}_{\ast}(Y)$ may also be made into a ring. Indeed, ${\rm CH}_{\ast}(Y)$ is the ring of $G$-invariants of ${\rm CH}_{\ast}(X)$. More precisely, if $V,W$ are subvarieties of $Y$, one may construct a refined intersection class $V\cdot W$ in ${\rm CH}_{\ast}(V\,\cap\,W)$, $m\,=\, \textrm{dim}, V\,+\, \textrm{dim}\,W\,-\, \textrm{dim}\,Y$, defined as 
\begin{equation*}
V\cdot W \ =\ (1/|G|)\eta_{\ast}(\pi_{\textrm{sp}}^{\ast}[V]\cdot \pi_{\textrm{sp}}^{\ast}[W]),   
\end{equation*}
where $\eta$ is the projection from $\pi^{-1}(V\,\cap\,W)$ to $V\,\cap\, W$. Note that the product on $Y$ is determined so that $\pi_{\textrm{sp}}^{\ast}(a\cdot b)\,=\, \pi_{\textrm{sp}}^{\ast}(a)\cdot \pi_{\textrm{sp}}^{\ast}(b)$,\,\, \,\textrm{and}\,\,\,$\pi_{\ast}(\pi_{\textrm{sp}}^{\ast}(a)\cdot c)\,=\, a\cdot \pi_{\ast}(c)$, for cycles $a,b$ on $Y$, $c$ on $X$, see \cite[Example 8.3.12]{Fu}.

%By \cite[17.4.10]{Fu}, the canonical homomorphism
%\begin{equation*}
%{\rm CH}^{\ast}(Y)\ \xlongrightarrow{\,\,\cap\,[Y]\,\,}\ {\rm CH}_{\ast}(Y)    
%\end{equation*}
%is an isomorphism of rings. This shows in particular that the ring structure on ${\rm CH}_{\ast}(Y)$
%is independent of $X$, and constructs pull-back homomorphism for arbitrary morphisms of such varieties. 

\subsubsection{Numerical equivalence}
\textbf{Singular case.}\ When $X$ is singular, we do not have an intersection pairing. Instead,
the Chern class action can be used. Following \cite[\S~19]{Fu}, we say that a $k$--cycle $Z$ is
\textit{numerically trivial}, denoted by $Z\,\equiv\, 0$, if 
\begin{equation*}
\textrm{deg}(P\cap [Z]_{Chow})\ =\ 0
\end{equation*}
for any weight $k$ homogeneous polynomial $P$ in the Chern classes of vector bundles on $X$; a Chern polynomial is
naturally seen as an operator on Chow groups using the linearity and commutativity of the action of
the Chern classes. The \textit{numerical group} of $k$--cycles is defined by $$N_{k}(X)\ =\ {\rm CH}_{k}(X)/\equiv .$$
It is a finite dimensional real vector space and is nonzero only for $0\,\leq\, k\,\leq \,\dim\, X\,=\, n$. The class in ${N}_{k}(X)$ of a real $k$--cycle $Z$ is denoted
by $[Z]$. Clearly, both $N_{0}(X)$ and $N_{n}(X)$ are isomorphic to $\mathbb{R}$.

The \textit{dual numerical group} $N^{k}(X)\,:=\, (N_{k}(X))^{\vee}$ is no longer isomorphic to
$N_{n-k}(X)$. It can also be defined as follows:
\begin{equation}\label{upper numerical groups}
{N}^{k}(X)\ = \ \frac{\text{Homogeneous Chern }\, \mathbb{R}\text{--polynomials $P$ of
weight }\, k}{\text{Chern polynomials }\, P\, \text{ such that }\, P\cap
\alpha \,=\,0\, \text{ for all }\, \alpha\,\in\, N_{k}(X)}.
\end{equation}
The multiplication of Chern polynomials induces a graded ring structure on $N^{\ast}(X)$. Moreover, the action of Chern classes induces linear maps
\[
N^{k}(X)\times N_m(X)
\longrightarrow
N_{m-k}(X),
\qquad
(P,\alpha)\longmapsto P\cap\alpha.
\]
There is a natural \textit{cyclification map}
\begin{equation}\label{isomorphism}
N^{k}(X) \ \longrightarrow \ N_{n-k}(X),
\end{equation}
defined by $P\,\longmapsto\,
P\,\cap\,[X]$, which is not in general an isomorphism. For $k\,=\,1$, the cyclification map $N^{1}(X) \, \longrightarrow \,
N_{n-1}(X)$ is injective; see
\cite[Example 19.3.3]{Fu}. Dually, $N^{n-1}(X) \, \longrightarrow \, N_{1}(X)$ is 
onto. More generally, $N_{\ast}(X)$ is a module over $N^{\ast}(X)$.

For notational simplicity, we will denote
\begin{equation}\label{P}
P \cap [X]\,\, \in\,\, N_k(X) \quad \text{by} \quad P.
\end{equation}

The space ${N}^{1}(X)$ is the N\'eron-Severi group of real Cartier divisors modulo numerical equivalence. Thus, ${N}_{1}(X)$ is the space of curves with
$\R$--coefficients modulo classes that have vanishing intersections against first Chern class
of invertible sheaves. The formal dual ${N}^{1}(X)$ is then the real space of the first Chern class of 
invertible sheaves modulo those having vanishing intersection against every curve.

\noindent\textbf{Smooth case.}\ When $X$ is smooth, there is an intersection pairing
\begin{equation}\label{e1}
{\rm CH}_{k}(X)\,\times\, {\rm CH}^{k}(X) \ \longrightarrow \ \mathbb{R}
\end{equation}
determined by the ring structure on ${\rm CH}^{\ast}(X)$ and the natural point counting degree
function 
$${\rm deg}\, :\, {\rm CH}_{0}(X)\, \longrightarrow \, \mathbb{R}.$$ 
The \textit{numerical group} $N_{k}(X)$ is the quotient 
of ${\rm CH}_{k}(X)$ by the kernel of this pairing; we denote $N^{k}(X) \,:=\, N_{n-k}(X)$. The pairing in \eqref{e1} 
induces a perfect pairing $N_{k}(X)\,\times\, N^{k}(X)\, \longrightarrow \, \mathbb{R}$, in particular, 
we have $N^{k}(X)\,\cong\, (N_{n-k}(X))^{\vee}$.

\begin{remark}
The definition of numerical equivalence for singular projective varieties given above agrees with the classical definition when $X$ is smooth. More precisely, if $X$ is smooth, then a cycle
$Z\,\in\, Z_{k}(X)$ is numerically trivial if and only if $[Z]\cdot\beta\, =\, 0$ for
all $\beta\,\in\, {N}_{n-k}(X)$, where the pairing is induced by the intersection product on the Chow ring; see $\emph{\cite[Example 19.1.5(a)]{Fu}}$.
\end{remark}

\subsection{Positive cones}
Throughout this subsection, $X$ is an arbitrary complex projective variety, not necessarily smooth.
\subsubsection{The pseudoeffective cone}
We say that a class $\alpha\,\in\, {N}_{k}(X)$ is \textit{effective} if $\alpha\, =\, [Z]$ for some effective $k$--cycle $Z$. Effective classes are closed under positive linear combinations, which motivates the following definition.

\begin{definition}
The closure of the convex cone generated by effective $k$--cycles on $X$ in ${N}_{k}(X)$ is denoted 
$\overline{\emph{Eff}}_{k}(X)$ and is called the \textit{pseudoeffective} cone. A class $\alpha\,\in\, 
{N}_{k}(X)$ is called \textit{pseudoeffective} (respectively, \textit{big}) if it lies in 
${\overline{\emph{Eff}}}_{k}(X)$ (respectively, in the interior of ${\overline{\emph{Eff}}}_{k}(X))$.
\end{definition}

Note that $\overline{\textrm{Eff}}_{1}(X)$ is referred to as the \textit{closed cone of curves} or the \textit{Mori cone of curves}, and is often
denoted by $\overline{\textrm{NE}}(X)$ in the literature.

\begin{definition}
A class $\beta\,\in \,{N}^{k}(X)$ is said to be \textit{pseudoeffective} if $\phi(\beta)\,\in\,
\overline{\textrm{Eff}}_{n-k}(X)$, where $\phi$ is the cyclification map defined in \eqref{isomorphism}. The resulting closed cone in $N^{k}(X)$ is denoted by $\overline{\emph{Eff}}^{\,k}(X)$ and is called the \textit{pseudoeffective dual cone}.
\end{definition}

\subsubsection{The nef cone}
Recall the perfect pairing $N^{k}(X)\times N_{k}(X)\,\longrightarrow \,\mathbb{R}$, defined by $(P,\,
\alpha)\,\longmapsto\, P\,\cap\,\alpha$.
\begin{definition}
 The \textit{nef cone} $\emph{Nef}^{\,\,k}(X)\,\subset\, N^{k}(X)$ is the dual of $\overline{\textrm{Eff}}_{k}(X)\,\subset\, N_{k}(X)$, with respect to the above perfect pairing.
\end{definition}
By definition, nefness is preserved under proper pullbacks. For simplicity, we will denote $\textrm{Nef}^{1}(X)$ by $\textrm{Nef}(X)$.

\section{Numerical groups and positive cones of quotients}

We begin by recalling some standard results related to finite group action on vector spaces.

For the following two results, \(\Bbbk\) denotes a field of characteristic $0$, unless otherwise specified.

\begin{lemma}\label{splits}
Let $G$ be a finite group, and let $U$, $V$, and $W$ be \(\Bbbk\)--vector spaces equipped with an action of $G$. Suppose
$$
0\, \longrightarrow\, U\, \longrightarrow\, V\, \xlongrightarrow{\pi}\, W \, \longrightarrow\, 0
$$
is a short exact sequence of $G$--modules over \(\Bbbk\). Then the sequence splits
as $G$--modules. Equivalently, there exists a $G$--equivariant map $\widetilde{s} \,:\, W\, \longrightarrow\, V$ such that $\pi\, \circ\, \widetilde{s} \,=\, \operatorname{Id}_W.$ Hence $V \,\cong\, U\, \oplus\, W$ as $G$--modules.
\end{lemma}

\begin{proof}
Since $W$ is a vector space over \(\Bbbk\), there exists a \(\Bbbk\)--linear splitting 
$s\, :\, W\, \longrightarrow\, V
$
such that
$\pi\, \circ\, s\, =\, \operatorname{Id}_W.
$
The map $s$ need not be $G$-equivariant. Define a new map
\[
\widetilde{s}\, :\, W \,\longrightarrow\, V
\]
by
\[
\widetilde{s}(w)
\,=\,
\frac{1}{|G|}
\sum_{g \in G} g \cdot s(g^{-1}w).
\]
Since $\operatorname{char}(\Bbbk)\,=\,0$, the scalar $\frac{1}{|G|}$ is
well-defined in \(\Bbbk\). We first show that $\widetilde{s}$ is $G$--equivariant. Let $h \,\in\, G$. Then
\[
\widetilde{s}(hw)
\,=\,
\frac{1}{|G|}
\sum_{g \,\in\, G} g \cdot s(g^{-1}hw).
\]
Putting $g \,=\, hk$, we obtain
\[
\widetilde{s}(hw)
\,=\,
\frac{1}{|G|}
\sum_{k \,\in\, G} hk \cdot s(k^{-1}w)
\,=\,
h \cdot
\left(
\frac{1}{|G|}
\sum_{k\, \in\, G} k \cdot s(k^{-1}w)
\right).
\]
Hence
\[
\widetilde{s}(hw)\,=\,h\cdot \widetilde{s}(w).
\]
Therefore $\widetilde{s}$ is $G$--equivariant. Since $\pi\, \circ\, s \,=\, \operatorname{id}_W$, this clearly implies that $\pi\,\circ\,\widetilde{s} \,=\, \operatorname{Id}_W$. Hence $\widetilde{s}$ is a $G$--equivariant splitting of the sequence, and so
$
V\, \cong\, U\, \oplus\, W
$
as $G$--modules.
\end{proof}

\begin{remark}
Let $\operatorname{char}(\Bbbk)\,>\,0$. Note that the above lemma remains valid whenever $\operatorname{char}(\Bbbk)\nmid |G|$.
\end{remark}

\begin{corollary}\label{canonical}
Let $V$ be a \(\Bbbk\)--vector space, and let $G$ be a finite group acting linearly on $V$. Let $V^{G}$
denote the subspace of $V$ consisting of elements that are invariant under the action of $G$.
\begin{enumerate}
\item[$(i)$] Suppose $W$ is a $G$--equivariant subspace of $V$. Then $W^{G}\,=\, V^{G}\,\cap\,W$ and there is a natural isomorphism
$({V}/{W})^{G}\,\cong\, {V^{G}}/{W^{G}}$.
\item[$(ii)$]  The inclusion map  $i\,:\,V^{G}\,\hookrightarrow\, V$ induces an isomorphism from $(V^{\vee})^{G}$ to $(V^{G})^{\vee}$.
\end{enumerate}
\end{corollary}

\begin{proof}
From Lemma \(\ref{splits}\), it follows that \((-)^{G}\) is an exact functor. Hence, the proof of \((i)\) follows. 

For the proof of \((ii)\), consider the inclusion
$i\,:\,V^{G}\,\hookrightarrow\, V.$ By Lemma \ref{splits}, the following short exact sequence of $G$ vector spaces splits:
\begin{equation*}
0\,\longrightarrow\, V^{G}\,\longrightarrow\, V\,\longrightarrow\, V/V^{G}\,\longrightarrow\,0.
\end{equation*}

There exists a $G$--subspace $W$ of $V$ such that we have $V\,\cong\, V^{G}\,\oplus\,W$ as $G$ vector spaces. After taking the dual, we obtain
\begin{equation*}
V^{\vee}\,\cong\, (V^{G})^{\vee}\,\oplus\, W^{\vee}.
\end{equation*}
Since \(G\) acts trivially on \(V^{G}\), it also acts trivially on \((V^{G})^{\vee}\). Therefore,
\begin{equation*}
((V^{G})^{\vee})^{G}\,=\,(V^{G})^{\vee}.
\end{equation*}
On the other hand,
\begin{equation*}
(V^{\vee})^{G}
\,=\,
((V^{G})^{\vee}\,\oplus\, W^{\vee})^{G}
\,=\,
(V^{G})^{\vee}\,\oplus\,(W^{\vee})^{G}.
\end{equation*}

To show the claim it remains to show that
\begin{equation*}
(W^{\vee})^{G}\,=\,0.
\end{equation*}
Let $f\,\in\,(W^{\vee})^{G}$. Then $g\cdot f\,=\,f$ for all $g\,\in\,G$. Assume that \(f\neq 0\). Then there exists \(w\in W\) such that $f(w)\,\neq\,0$. Consider the  vector $w_{0}\,:=\,\sum_{g\in G}gw$. Then clearly $w_{0}\,\in\,W^{G}$. Clearly, \(W^{G}=0\), it follows that $w_{0}\,=\,0$. Since $\textrm{char}(\Bbbk)\,=\,0$ and $|G|\,\neq\,0$, we conclude $f(w)\,=\,0$. Hence $f\,=\,0$, and therefore $(W^{\vee})^{G}\,=\,0$.
\end{proof}

The following result of Dang asserts that a flat pullback map on the Chow groups descends to a map 
between the numerical groups and, dually, the pushforward to numerical dual groups.

\begin{proposition}[{\cite[Corollary 9.5]{Da}}]\label{da}
Let $q:\, X\,\longrightarrow\,Y$ be a flat morphism of relative dimension $e$ between normal projective
varieties. Then the morphism $q^{\ast}:\, {\rm CH}_{\ast}(Y)\,\longrightarrow\,{\rm CH}_{e+\ast}(X)$ induces a
morphism of abelian groups $q^{\ast}\,:\, N_{\ast}(Y)\,\longrightarrow\, N_{e+\ast}(X)$. By
duality, the morphism $q_{\ast}\,:\, {\rm CH}^{\ast}(X)\,\longrightarrow\,{\rm CH}^{\ast-e}(Y)$ induces a morphism of
abelian groups $q_{\ast}\,:\, N^{\ast}(X)\,\longrightarrow\, N^{\ast-e}(Y)$.
\end{proposition}
We first show that the special pullback is compatible with the usual pullback of Cartier divisors.
\begin{proposition}\label{F}
Let $X$ be a normal complex projective variety of dimension $n$, and let $G$ be a finite group acting on $X$. Let $\pi\,:\,X\,\longrightarrow\, X/G$ be the quotient map. Let $D$ be an integral Cartier divisor on $X/G$, and let $[D] \,\in\, {{Z}}_{n-1}(X/G)$ be its associated Weil divisor cycle. Then
\[
\pi^{\ast}_{\emph{sp}}([D]) \ =\ [\,\pi^\ast D\,],
\]
i.e., the Weil divisor associated with the Cartier divisor $\pi^\ast D$
coincides with the special pullback of the Weil divisor associated with $D$.
\end{proposition}

\begin{proof}
Since every Cartier divisor can be written as a difference of
effective Cartier divisors, it is enough to prove the statement when
$D$ is effective. Write
\begin{equation*}
[D]\,=\,\sum_{j=1}^{s} b_jD_j,
\end{equation*}
where the $D_j$ are distinct prime divisors on $X/G$ and $b_j\,>\,0$. Fix a component $D_0\,:=\,D_j$, and set $b\,:=\,b_j$. Let $\xi$
denote the generic point of $D_0$. Since $D$ is Cartier, there
exists an open neighbourhood $U'$ of $\xi$ on which $D$ is
principal. Choose a sufficiently small affine open neighbourhood
\begin{equation*}
U\,\subseteq\,
U'\,\setminus\,\bigcup_{\ell\neq j}D_\ell
\end{equation*}
of $\xi$. Then
\begin{equation*}
[D|_U]\,=\,b(D_0\,\cap\, U).
\end{equation*}
Moreover, since $D|_U$ is the restriction of the original Cartier
divisor $D$, there exists $f\,\in\,\mathcal O_{X/G}(U)$ such that $D|_U\,=\,Z(f).$ Replacing $X/G$ by $U$, $X$ by $\pi^{-1}(U)$, and $D$ by
$D|_U$, we may therefore assume that $X/G$ is affine and write
\begin{equation*}
[D]\,=\,bD_0,
\qquad
D\,=\,Z(f),
\end{equation*}
where $D_0$ is a prime divisor on $X/G$,
$f\,\in\,\mathcal O_{X/G}(X/G)$, and $b\,>\,0$.
Let
\begin{equation*}
\pi^{-1}(D_0)\,=\, \bigcup_{i=1}^r E_i
\end{equation*}
be the decomposition into irreducible components.  
Let $\xi$ denote the generic point of $D_0$, and set
\begin{equation*}
A\, :=\, \mathcal{O}(X)_{\xi}, 
\end{equation*}
so $A^G\, =\, \mathcal{O}(X/G)_{\xi}$.

Since $D_{0}$ is a prime divisor and $X/G$ is normal , $A^{G}$ is a DVR. Moreover, $A$ is a semilocal ring whose maximal ideals $\xi_{1},..., \xi_{r}$ correspond to the divisors $E_{1},...,E_{r}$.

Since the action of $G$ permutes the divisors $E_{i}$ transitively, the residue field extensions 
\begin{equation*}
k(E_{i})/k(D_{0})    \end{equation*}
all have the same degree. Set

\begin{equation*}
 d\,:=\, [k(E_i):\, k(D_0)].  
\end{equation*}

So, for any simple $A$-module $M$,
\begin{equation*}
\ell_{A}(M)\,=\, \frac{\ell_{A^{G}}(M)}{[k(E_{i}):\,k(D_{0})]},    
\end{equation*}
hence it is true for all $A$-modules $M$ of finite length.

So,
\begin{align*}
\textrm{ord}_{E_i}(f) & =\, \ell_{A_{\xi_i}}(A_{\xi_i}/(f)) =\, \frac{1}{r}\sum_{j=1}^r \ell_{A_{\xi_j}}(A_{\xi_j}/(f)) \quad(\textrm{since $G$ acts transitively on the $\xi_j$'s}\,\ )\\
&\hspace{2.5cm} =\, \frac{1}{r}\sum_{j=1}^r \frac{1}{d}\ell_{A^G}(A_{\xi_j}/(f)) \quad(\textrm {see \cite[Lemma A.1.3]{Fu})}\\
& \hspace{2.5cm} =\, \frac{1}{rd} \sum_{j=1}^r \ell_{A^G}(A_{\xi_j}/(f)) = \, \frac{1}{rd} \ell_{A^G}(A/(f)) \quad(\textrm{since $A/(f)\,\cong\, \bigoplus_{j=1}^{r} A_{\xi_j}/(f)$})\\
&\hspace{2.5cm} =\, \frac{1}{r\,d}\, \ell_{A^G}(A^G/(f))\,[K(A):K(A^G)] \quad(\textrm{see \cite[Example A.3.3]{Fu})}\\
&\hspace{2.5cm} =\, \frac{b\,|G|}{r\,d} \quad(\textrm{since}\, [K(A):K(A^G)] \,=\, |G|),
\end{align*}
where $K(A)$ and $K(A^{G})$ denote the function fields of $A$ and $A^{G}$, respectively.

Hence,
\begin{equation}\label{ord}
[\pi^\ast D]
    \ =\ \sum_{i=1}^r \textrm{ord}_{E_i}(f)\,E_i
   \  =\ \frac{b\,|G|}{r\,d} \sum_{i=1}^r E_i.
\end{equation}

Now define
\[
H_i\ :=\ \{\, g \,\in\, G\,\,\big\vert\,\, g(E_i) \,=\, E_i \}.
\]
By \cite[Lemma~15.111.9]{SP}, the extension $k(E_i)/k(D_0)$ is normal and hence Galois. Furthermore, the
homomorphism
\[
H_i\ \longrightarrow\ \textrm{Aut}(k(E_i)/k(D_0))
\]
is surjective with kernel $I_{E_i}$.  Note that $|I_{E_1}|\, =\, \ldots\, =\, |I_{E_r}|
\, =:\, d^{\prime}$ (say). Therefore,
\begin{equation*}
 \left|\textrm{Aut}(k(E_{i})/k(D_{0}))\right|\ =\ |H_{i}|/|I_{E_{i}}|,   
\end{equation*}
and for each $i$, we have
\begin{equation}\label{r}
d\, =\, [k(E_{i}):\, k(D_{0})]\,=\, \left|\textrm{Aut}(k(E_{i})/k(D_{0}))\right|\,=\, |H_{i}|/ |I_{E_{i}}|
\,=\, |H_{i}|/d^{\prime}.
\end{equation}
Note that $r\,=\, |G|/|H_{i}|$. Now substituting \eqref{r} in \eqref{ord} it is deduced that
\begin{equation*}
[\pi^{\ast}D]\ =\ b d^{\prime}\,\sum_{i=1}^{r}E_{i}\ =\ b \sum_{i=1}^{r} |I_{E_i}|E_{i}.
\end{equation*}
Thus, we have $\pi^{\ast}_{\textrm{sp}}([D])\,=\,[\pi^{\ast}D]$.
\end{proof}

\begin{proposition}
Let $X$ be a normal complex projective variety, and let $G$ be a
finite group acting on $X$. Let $\pi\,\colon\, X\,\longrightarrow\, X/G$ be the quotient map, and assume that $\pi$ is flat. Then
$\pi^{\ast}
\,=\,
\pi^{\ast}_{\mathrm{sp}}
\,\colon\,
Z_{\ast}(X/G)\,\longrightarrow\, Z_{\ast}(X).$
\end{proposition}

\begin{proof}
By our assumption, $\pi\, :\, X\, \longrightarrow \,X/G$ is flat, so the flat pullback
$\pi^{\ast} \,:\, Z_{\ast}(X/G)\, \longrightarrow\, Z_{\ast}(X)$ is well-defined.
By \cite[Examples~1.7.4 and~1.7.6]{Fu}, we have
\begin{equation}\label{speq}
\pi_{\ast} \,\circ\, \pi^{\ast}
\ =\ \pi_{\ast} \,\circ\, \pi^{\ast}_{\mathrm{sp}}
\ =\ |G|\,\mathrm{Id}.
\end{equation}

Let $V$ be a subvariety of $X/G$, and let $W_{1},\,\cdots,\,W_{r}$ be the irreducible components
of $\pi^{-1}(V)$. Since $G$ acts transitively on these
irreducible components, we have, for any $i$,
\[
e \ =\ \ell_{\mathcal{O}_{X}}\!\bigl(\mathcal{O}_{W_i}\bigr).
\]
Similarly, the inertia groups of the components are conjugate to each other, and we have
\[
m \, =\, \textrm{card}(I_{W_i})
 \quad \textrm{and} \quad  [K(W_{1})\,:\,K(V)]\,=\, [K(W_{i})\,:\,K(V)]  \quad\text{for all}\,\, i.
\]

By definition
\[
\pi^{\ast}[V] \;=\; e \sum_{i=1}^{r} [W_i], \,\ \textrm{and} \,\ \pi^{\ast}_{\mathrm{sp}}[V] \;=\; m \sum_{i=1}^{r} [W_i].
\]

Hence, by \eqref{speq}, we have $e\,=\,m$. Therefore, $\pi^{\ast}[V]
\,=\,
\pi^{\ast}_{\mathrm{sp}}[V].$ Since $V$ was arbitrary, we conclude that $\pi^{\ast}
\,=\,
\pi^{\ast}_{\mathrm{sp}}.$
\end{proof}

\noindent
\textbf{Assumption:} 
From now on, throughout the rest of the article, we assume that $\pi\,:\, X\,\longrightarrow\, X/G$ is a flat map.

Note that, the hypotheses of Proposition \ref{da} are satisfied for the morphism $\pi$. Thus, we have the following maps:
\begin{equation}\label{g} 
N_{\ast}(X/G)\ \xlongrightarrow{\pi^{\ast}}\ N_{\ast}(X)\ \xlongrightarrow{\pi_{\ast}}\ N_{\ast}(X/G).
\end{equation}

\begin{theorem}\label{lower}
There is an isomorphism $N_{\ast}(X/G)\,\cong\, N_{\ast}(X)^{G}$, induced by $\pi^{\ast}$, and furthermore,
the composition of maps $\pi_{\ast}\,\circ\,\pi^{\ast}$ in \eqref{g} coincides with $|G|\cdot\emph{Id}$.
\end{theorem}

\begin{proof}
By \cite[Example 1.7.6]{Fu}, the composition of maps
\begin{equation*}
Z_{\ast}(X/G)\,\xlongrightarrow{\pi^{\ast}}\,Z_{\ast}(X)^{G}\,\hookrightarrow\, Z_{\ast}(X) \,\xlongrightarrow{\pi_{\ast}}\,Z_{\ast}(X/G)
\end{equation*}
is $|G|\cdot\textrm{Id}$. Hence, we have the following commutative diagram: 
\begin{center}
\begin{tikzcd}
Z_{\ast}(X/G) \arrow[r, "\pi^{\ast}"] \arrow[d] & Z_{\ast}(X) \arrow[r, "\pi_{\ast}"] \arrow[d] \arrow[dr, phantom, ""] & Z_{\ast}(X/G) \arrow[d] \\
N_{\ast}(X/G) \arrow[r, "\pi^{\ast}"] &N_{\ast}(X)\arrow[r, "\pi_{\ast}"] & N_{\ast}(X/G).
\end{tikzcd}
\end{center} 
Furthermore, in \cite[Example 1.7.6]{Fu}, it has been shown that $Z_{\ast}(X/G)\,\cong\, Z_{\ast}(X)^{G}$. Thus by above diagram it is clear that
\begin{equation*}
\pi^{\ast}(\textrm{Num}(X/G))\ \subset\ Z_{\ast}(X)^{G}\,\cap\,\textrm{Num}(X).
\end{equation*}

Clearly, by Corollary \ref{canonical}, we have $N_{\ast}(X)^{G}\,=\,Z_{\ast}(X)^{G}/(Z_{\ast}(X)^{G}\,\cap\, 
\textrm{Num}(X))$. Hence, the following commutative diagram is obtained:
\begin{center}
\begin{tikzcd}
Z_{\ast}(X/G) \arrow[r, "\pi^{\ast}"] \arrow[d] & Z_{\ast}(X)^{G} \arrow[r, "\pi_{\ast}"] \arrow[d] \arrow[dr, phantom, ""] & Z_{\ast}(X/G) \arrow[d] \\
N_{\ast}(X/G) \arrow[r, "\pi^{\ast}"] &N_{\ast}(X)^{G}\arrow[r, "\pi_{\ast}"] & N_{\ast}(X/G).
\end{tikzcd}
\end{center}
Therefore, the composition of maps in the bottom row of the above diagram
coincides with $|G|\cdot\textrm{Id}$. Thus, we conclude the pullback and the pushforward maps
$\pi^{\ast}\,:\,N_{\ast}(X/G)\,\longrightarrow\, N_{\ast}(X)^{G}$ and $\pi_{\ast}\,
:\, N_{\ast}(X)^{G}\,\longrightarrow\, N_{\ast}(X/G)$ are isomorphisms.
\end{proof}

\begin{remark}
Note that $(\pi^{\ast})^{-1}\, =\, \frac{1}{|G|}\,\pi_{\ast}$, and similarly $(\pi_{\ast})^{-1}\, =\, \frac{1}{|G|}\,\pi^{\ast}$.
\end{remark}

\begin{corollary}\mbox{}
\begin{enumerate}
\item[$(i)$] The composition of maps $$N_{\ast}(X)^G\ \xlongrightarrow{\,\pi_{\ast}\,}\
 N_{\ast}(X/G) \ \xlongrightarrow{\,\pi^{\ast}\,}\ N_{\ast}(X)^G$$ equals \( |G| \cdot \mathrm{Id} \).

\item[$(ii)$] The composition of maps 
$$N_{\ast}(X)\ \xlongrightarrow{\,\pi_{\ast}\,}\ N_{\ast}(X/G)\ \xlongrightarrow{\,\pi^{\ast}\,}\ N_{\ast}(X)$$
satisfies the following:
$$\pi^{\ast} \pi_{\ast}(\alpha)\ =\ \sum_{g \in G} g \cdot \alpha$$
for all \( \alpha \,\in\, N_{\ast}(X) \).
\end{enumerate}
\end{corollary}

\begin{proof} Statement $(i)$: This follows immediately from Theorem \ref{lower}.

Proof of $(ii)$:\,\, Since $\pi$ is $G$-invariant, for $\alpha\,\in\,N_{\ast}(X)$ we have:
\begin{equation*}
\pi^{\ast}\pi_{\ast}\big(g\cdot\alpha\big)\ =\ \pi^{\ast}\pi_{\ast}(g_{\ast}\alpha)\ =\
\pi^{\ast}\pi_{\ast}\alpha
\end{equation*}
for all $g\,\in\,G$. Therefore,
\begin{equation*}
\pi^{\ast}\pi_{\ast}\big(\sum_{g\,\in\, G}g\cdot\alpha \big)\,=\,
\sum_{g\,\in\, G}\pi^{\ast}\pi_{\ast}(g\cdot\alpha)\,=\,\sum_{g\in G}\pi^{\ast}\pi_{\ast}\alpha\,=\,|G|\,\pi^{\ast}\pi_{\ast}\alpha. 
\end{equation*}
As $\sum_{g\,\in\, G}g\cdot\alpha$ is a $G$--invariant cycle, combining $(i)$ with the above equality we obtain: 
\begin{equation*}
 \pi^{\ast}\pi_{\ast}\alpha = \sum_{g\,\in\,G}g\cdot \alpha.    
\end{equation*}
\end{proof}
As a consequence of the above results, the dual numerical groups are also isomorphic. More precisely, we have the following:

\begin{corollary}\label{invariant}
\begin{enumerate}
\item[$(i)$] $$N^{\ast}(X/G)\ \cong\ N^{\ast}(X)^{G}$$ and the composition of maps
\begin{equation*}
N^{\ast}(X/G)\ \xlongrightarrow{\,\pi^{\ast}\,}\ N^{\ast}(X)^{G}\ \xlongrightarrow{\,\pi_{\ast}\,}\
N^{\ast}(X/G) 
\end{equation*}
equals $|G|\cdot\emph{Id}$.

\item[$(ii)$] The composition of maps $$N^{\ast}(X)^G\ \xlongrightarrow{\,\pi_{\ast}\,}\
 N^{\ast}(X/G) \ \xlongrightarrow{\,\pi^{\ast}\,}\ N^{\ast}(X)^G$$ equals \( |G| \cdot \mathrm{Id} \).
\end{enumerate}

\end{corollary}

\begin{proof}
By Corollary \ref{canonical}, we have a natural isomorphism:
\begin{equation*}
(N_{\ast}(X)^{G})^{\vee}\ =\ (N_{\ast}(X)^{\vee})^{G}\ =\ N^{\ast}(X)^{G}.
\end{equation*}
Using this isomorphism and taking the dual of 
\begin{equation*}
N_{\ast}(X/G)\ \xlongrightarrow{\,\pi^{\ast}\,}\ N_{\ast}(X)^{G}\ 
\xlongrightarrow{\,\pi_{\ast}\,}\ N_{\ast}(X/G), 
\end{equation*}
the following sequence is obtained:
\begin{equation*}
N^{\ast}(X/G)\ \xlongrightarrow{\,\pi^{\ast}\, }\ N^{\ast}(X)^{G}\
\xlongrightarrow{\,\pi_{\ast}\, }\ N^{\ast}(X/G).
\end{equation*}
Clearly, the composition of the above maps coincides with $|G|\cdot\textrm{Id}$. Thus, $N^{\ast}(X/G)\,
\cong\, N^{\ast}(X)^{G}$. The next statement may be proved analogously. 
\end{proof}

Recall that we have pushforward map $\pi_{\ast}\,:\, \overline{\textrm{Eff}}_{k}(X)\ \xlongrightarrow\ \overline{\textrm{Eff}}_{k}(X/G)$ (see \cite[Corollary 3.22]{FL}) induced by $\pi \,:\, X \,\to\, X/G$, which is continuous.

\begin{theorem}\label{pullback}
The pullback map $\pi^{\ast}\,:\, N_{\ast}(X/G)\, \longrightarrow\, N_{\ast}(X),$ when restricted to the pseudoeffective cone ${\overline{\emph{Eff}}_{k}(X/G)}$, induces a map 
${\overline{\emph{Eff}}_{k}(X/G)}\ \xlongrightarrow{\,\pi^{\ast}\, }\ \overline{\emph{Eff}}_{k}(X)$. Furthermore, we have isomorphisms of cones ${\overline{\emph{Eff}}_{k}(X/G)}\ \xlongrightarrow{\,\pi^{\ast}\, }\ \overline{\emph{Eff}}_{k}(X)^{G}$, and $\overline{\emph{Eff}}_{k}(X)^{G}\ \xlongrightarrow{\,\pi_{\ast}\, }\ {\overline{\emph{Eff}}_{k}(X/G)}$ where $\overline{\emph{Eff}}_{k}(X)^{G} = \overline{\emph{Eff}}_{k}(X)\, \cap\, N_k(X)^{G}$.
\end{theorem}

\begin{proof}
Clearly, the pullback map $\pi^{\ast}$ is continuous. By definition, the pullback of an effective cycle is effective.
Using the continuity of $\pi^{\ast}$ it follows that if $\alpha$ lies in the closed cone
$\overline{\textrm{Eff}}_{k}(X/G)$, then the pullback $\pi^{\ast}(\alpha)$ lies in
$\overline{\textrm{Eff}}_{k}(X)$. Moreover, restricting the composition of maps in \eqref{g}, we obtain the following:
\begin{equation*}
\overline{\textrm{Eff}}_{k}(X/G)\ \xlongrightarrow{\,\pi^{\ast}\,}\ \overline{\textrm{Eff}}_{k}(X)\
\xlongrightarrow{\,\pi_{\ast}\, }\ \overline{\textrm{Eff}}_{k}(X/G),
\end{equation*}
and again this composition of maps agrees with multiplication by $|G|\cdot \textrm{Id}$. Consequently $\pi^{\ast}$ is injective and $\pi_{\ast}$ is surjective. Clearly the pullback $\pi^{\ast}$ factors through $\overline{\textrm{Eff}}_{k}(X)^G \,\subset\, \overline{\textrm{Eff}}_{k}(X)$. 

Now consider $\alpha \in \overline{\textrm{Eff}}_{k}(X)^G$ and assume that $\alpha_i$ are effective classes converging to $\alpha$. Replacing $\alpha_i$ with $\frac{1}{|G|} \sum_{g\,\in\,G}g\cdot\alpha_i$, if necessary, we may assume that $\alpha_i\, \in\, N_k(X)^G$. By Theorem \ref{lower}, we have $\beta_i \,\in\, N_k(X/G)$ such that $\pi^{\ast}(\beta_i)\, =\, \alpha_i$ for all $i$. Observe that $\pi_{\ast}\pi^{\ast}(\beta_i) = |G| \beta_i$, and hence $\beta_i = \frac{1}{|G|}\pi_{\ast}(\alpha_i)$. As $\alpha_i$ is effective, we deduce that $\beta_i$ must be effective for all $i$. As $\pi_{\ast}\, :\, N_{\ast}(X)^G \,\to\, N_{\ast}(X/G)$ is a continuous homomorphism, it follows that $\beta_i$ must converge to a class $\beta$. Clearly $\beta \,\in\, \overline{\textrm{Eff}}_{k}(X/G)$ and $\pi^{\ast}(\beta) \,=\, \alpha$. Thus, we deduce that $\pi^{\ast}$ is a surjection and hence a bijection. Consequently, the same is true for $\pi_{\ast}$. 
\end{proof}

The following lemma will be needed in the proof of the corresponding result for the pullback and pushforward
maps between $N^{\ast}(X/G)$ and $N^{\ast}(X)$.

\begin{lemma}\label{distributes}
For any class $P\,\in\, N^{\ast}(X)$,
$$\pi_{\ast}(\phi_{X}(P))\ =\ \phi_{X/G}(\pi_{\ast}(P)).$$ More
precisely, $\pi_{\ast}(P\,\cap\,[X])\,=\,\frac{1}{|G|}(\pi_{\ast}(P)\,\cap\,\pi_{\ast}[X])$.
\end{lemma}

\begin{proof}
Take $P\,\in\,N^{\ast}(X)$. Consider the $G$-invariant class $\sum_{g\in G}g\cdot P$. By Corollary 
\ref{invariant}, this sum descends to a class $Q\,\in\, N^{\ast}(X/G)$ such that
\begin{equation*}
\pi^{\ast}(Q)\ =\ \sum_{g\in G}g\cdot P.
\end{equation*}
As seen earlier it can be seen that

\begin{equation*}
\pi_{\ast}P\ =\ Q.
\end{equation*}
Now, we observe that
\begin{equation}
\pi_{\ast}\left(\sum_{g\in G} g\cdot P \cap [X]\right)\,=\, \pi_{\ast}(\pi^{\ast}Q\cap\pi^{\ast}[X/G])\,=\, \pi_{\ast}\pi^{\ast}(Q\,\cap\, [X/G])\,=\,|G|(Q\,\cap\, [X/G]). 
\end{equation}
On the other hand, we also have
\begin{equation*}
\pi_{\ast}\left(\sum_{g\in G} g\cdot P\right)\,\cap\, \pi_{\ast}[X]\,=\,|G|Q \,\cap\, |G|\big[X/G\big]
\end{equation*}
\begin{equation*}
\hspace{4.3cm} = |G|^{2} \big(Q\,\cap\, \big[X/G\big]\big).
\end{equation*}
So,
\begin{equation}\label{3}
\pi_{\ast}\left(\sum_{g\in G} g\cdot P \cap [X]\right)\,=\, \frac{1}{|G|}\left(\pi_{\ast}\left(\sum_{g\in G}g\cdot P\right)\cap \pi_{\ast}[X]\right). 
\end{equation}
Now observe that the cap product is compatible with pullback under group action:
\begin{equation*}
g\cdot P\,\cap\, [X]\,=\, g^{\ast}P\,\cap\,g^{\ast}[X]\,=\, g^{\ast}(P\,\cap\,[X]). 
\end{equation*}
So
\begin{equation*}
\sum_{g\in G} g\cdot P\,\cap\, [X]\,=\, \sum_{g\in G}g^{\ast}(P\,\cap\,[X]).
\end{equation*}
Then, applying $\pi_{\ast}$ and using the $G$-invariance of $\pi$,
\begin{equation}\label{4}
\pi_{\ast}\left(\sum_{g\in G}g\cdot P\,\cap\, [X]\right)\,=\, \sum_{g\in G}\pi_{\ast}(g^{\ast}(P\,\cap\, [X]))\,=\, |G|(\pi_{\ast}(P\,\cap\,[X])). 
\end{equation}
Combining \eqref{3} and \eqref{4}, we conclude that
\begin{equation*}
 |G|\pi_{\ast}(P\,\cap\,[X])\ =\
\frac{1}{|G|}\left(\pi_{\ast}\left(\sum_{g\in G}g\cdot P\right)\,\cap\,\pi_{\ast}[X])\right).
\end{equation*}
Combining this with the fact that $\pi_{\ast}\left(\sum_{g\in G}g\cdot P\right)\,=\, |G|\pi_{\ast}(P)$,
\begin{equation*}
 \pi_{\ast}(\phi_{X}(P))\,=\,\pi_{\ast}(P\,\cap\,[X])\ =\ \frac{1}{|G|}\left((\pi_{\ast}(P)\,\cap\, \pi_{\ast}[X]\right))\,=\,\phi_{X/G}(\pi_{\ast}(P)),
 \end{equation*}
which proves the lemma.
\end{proof}

\begin{theorem}\label{pseu}
The pullback and pushforward maps between $N^{\ast}(X)$ and $N^{\ast}(X/G)$ restrict to morphism of
pseudoeffective dual cones as follows:
\begin{enumerate}
\item[$(i)$] The pullback map $\pi^{\ast}\,:\, N^{\ast}(X/G)\, \longrightarrow\, N^{\ast}(X)$ restricts to
a map between the cones of pseudoeffective dual classes
\begin{equation*}
{\overline{\emph{Eff}}^{\,k}(X/G)}\ \xlongrightarrow{\,\pi^{\ast}\,}\ \overline{\emph{Eff}}^{\,k}(X),
\end{equation*}
which induces an isomorphism $\overline{\emph{Eff}}^{\,k}(X/G)\ \xlongrightarrow{\,\pi^{\ast}\,}\ \overline{\emph{Eff}}^{\,k}(X)^G$ for all $k\,\in\, \{1,\,\cdots,\,n\}$, where $n\,=\, \dim X$.

\item[$(ii)$] The pushforward map $\pi_{\ast}\ :\ N^{\ast}(X)\ \longrightarrow\ N^{\ast}(X/G)$ restricts to a map between the pseudoeffective dual classes
\begin{equation*}
\overline{\emph{Eff}}^{\,k}(X)\ \xlongrightarrow{\pi_{\ast}}\ \overline{\emph{Eff}}^{\,k}(X/G)
\end{equation*}
which induces an isomorphism $\overline{\emph{Eff}}^{\,k}(X)^G\ \xlongrightarrow{\pi_{\ast}}\ \overline{\emph{Eff}}^{\,k}(X/G)$
for all $k\,\in\, \{1,\,\cdots,\, n\}$, where $n\,=\, \dim X$.
\end{enumerate}    
\end{theorem}

\begin{proof}
Consider the cyclification maps
\begin{equation*}
N^{k}(X)\,\times\, N_{n}(X)\,\xlongrightarrow{\phi_{X}}\,N_{n-k}(X)\quad \textnormal{and} \quad N^{k}(X/G)\,\times\, N_{n}(X/G)\,\xlongrightarrow{\phi_{X/G}}\,N_{n-k}(X/G),
\end{equation*}
defined for $X$ and $X/G$ respectively. Take any $\alpha\,\in\,\overline{\textrm{Eff}}^{\,k}(X/G)$. To prove
that $\pi^{\ast}(\alpha)\,\in\,\overline{\textrm{Eff}}^{\,k}(X),$ it suffices to show that
$\pi^{\ast}(\alpha)\,\cap\, [X]\,\in\, \overline{\textrm{Eff}}_{n-k}(X)$. Since $\alpha\,\in\,\overline{\textrm{Eff}}^{\,k}(X/G)$, by definition,
\begin{equation*}
\phi_{X/G}(\alpha)\,=\, \alpha\,\cap\, [X/G]\,\in\,\overline{\textrm{Eff}}_{n-k}(X/G). 
\end{equation*}
Now consider
\begin{equation*}
\phi_{X}(\pi^{\ast}(\alpha))\,=\, \pi^{\ast}(\alpha)\,\cap\,[X]\,=\, \pi^{\ast}(\alpha)\,\cap\,\pi^{\ast}[X/G]
\,=\, \pi^{\ast}(\alpha\, \cap\, [X/G]),
\end{equation*}

where the last equality follows from \textnormal{\cite[Theorem 3.2(d)]{Fu}}. From Theorem \ref{pullback}, we know that the pullback
\begin{equation*}
\overline{\textrm{Eff}}_{n-k}(X/G)\ \xlongrightarrow{\,\pi^{\ast}\, }\ \overline{\textrm{Eff}}_{n-k}(X)
\end{equation*}
is a well-defined map. Therefore, $\pi^{\ast}(\alpha\,\cap\,[X/G])\,\in\, \overline{\textrm{Eff}}_{n-k}(X)$,
and hence $\pi^{\ast}(\alpha)\,\in\, \overline{\textrm{Eff}}^{\,k}(X)$. Thus, we have a well-defined map
$\overline{\textrm{Eff}}^{\,k}(X/G)\,\xlongrightarrow{\,\pi^{\ast}\,}\, \overline{\textrm{Eff}}^{\,k}(X)$. Clearly the image of $\pi^{\ast}$ lies in $\overline{\textrm{Eff}}^{\,k}(X)^G$.

On the other hand, take $P\,\in\,\overline{\textrm{Eff}}^{\,k}(X)$. By definition,  $P\,\cap\,[X]\,\in\,
 \overline{\textrm{Eff}}_{n-k}(X)$. By Lemma \ref{distributes},
 \begin{equation*}
 \pi_{\ast}(P\,\cap\,[X])\ =\ \frac{1}{|G|}(\pi_{\ast}(P)\,\cap\, \pi_{\ast}[X])
 \ =\ \pi_{\ast}(P)\,\cap\, [X/G].
 \end{equation*}
 As $\pi_{\ast}(P\,\cap\,[X]) \in \overline{\textrm{Eff}}_{n-k}(X/G)$, it follows from definition that $\pi_{\ast}(P)\,\in\,\overline{\textrm{Eff}}^{\,k}(X/G)$. In particular, by restricting $\pi_{\ast}$, we get a map $\overline{\textrm{Eff}}^{\,k}(X)^G\ \xlongrightarrow{\,\pi_{\ast}\,}\ \overline{\textrm{Eff}}^{\,k}(X/G)$.

Thus, on restricting the composition of maps in Corollary \ref{invariant}, it follows that the
 composition of maps
 \begin{equation*}
 \overline{\textrm{Eff}}^{\,k}(X/G)\ \xlongrightarrow{\,\pi^{\ast}\, }\
 \overline{\textrm{Eff}}^{\,k}(X)^G\ \xlongrightarrow{\,\pi_{\ast}\,}\ \overline{\textrm{Eff}}^{\,k}(X/G)
 \end{equation*}
and,
\begin{equation*}
 \overline{\textrm{Eff}}^{\,k}(X)^G\ \xlongrightarrow{\,\pi_{\ast}\, }\
 \overline{\textrm{Eff}}^{\,k}(X/G)\ \xlongrightarrow{\,\pi^{\ast}\,}\ \overline{\textrm{Eff}}^{\,k}(X)^G
 \end{equation*}
 equals multiplication by $|G|$. The proof follows immediately.
\end{proof}

\begin{proposition}\label{nef}
The pullback $\pi^{\ast} : \emph{Nef}^{\,k}(X/G)\,\longrightarrow\, \emph{Nef}^{\,k}(X)^G$ and the pushforward $\pi_{\ast} : \emph{Nef}^{\,k}(X)^G\,\longrightarrow\, \emph{Nef}^{\,k}(X/G)$ are isomorphisms.  

\end{proposition}
\begin{proof}
We have an isomorphism $\pi^{\ast} \,:\, \overline{\textrm{Eff}}_{\,k}(X/G)\, \longrightarrow\,
\overline{\textrm{Eff}}_{\,k}(X)^G$ (see Theorem \ref{pullback}). By taking dual, we deduce that $\pi_{\ast}$ is an isomorphism between the Nef cones. The other claim follows similarly from noting that $\pi_{\ast} \,:\, \overline{\textrm{Eff}}_{k}(X)^G\ \longrightarrow\
\overline{\textrm{Eff}}_{k}(X/G)$ is an isomorphism.    
\end{proof}

\begin{proposition}\label{equality}
The following two statements are equivalent for all $k\,\in\, \{1,\,\cdots,\,n-1\}$:
\begin{enumerate}
\item[(i)] $\overline{\emph{Eff}}^{\,k}(X)^G\ =\ \emph{Nef}^{\,k}(X)^G$.

\item[(ii)] $\overline{\emph{Eff}}^{\,k}(X/G)\ =\ \emph{Nef}^{\,k}(X/G)$.
\end{enumerate}
\end{proposition}
\begin{proof}
Assume that $\overline{\textrm{Eff}}^{\,k}(X)^G\,\subset\, \textrm{Nef}^{\,k}(X)^G$. Consider the following commutative diagram:
\begin{equation*}
\begin{tikzcd} 
\overline{\textrm{Eff}}^{\,k}(X/G) \arrow[d, "\pi^{\ast}"] \arrow[r, hook, dashed]
 & \textrm{Nef}^{\,k}(X/G) \\
\overline{\textrm{Eff}}^{\,k}(X)^G \arrow[r, hook] & \textrm{Nef}^{\,k}(X)^G \arrow[u, swap, "\pi_{\ast}"] 
\end{tikzcd}
\end{equation*}
Take any \( P \,\in\, \overline{\textrm{Eff}}^{\,k}(X/G) \). Then $\pi^{\ast}(P)\, \in\,
\overline{\textrm{Eff}}^{\,k}(X)^G$, and by assumption, also lies in \( \textrm{Nef}^{\,k}(X)^G \). From
the fact that $\pi_{\ast}(\pi^{\ast}(P)) \,=\, |G| \cdot P\,\in\, \textrm{Nef}^{\,k}(X/G)$
we have \( P\, \in\, \textrm{Nef}^{\,k}(X/G) \), and hence
\[
\overline{\textrm{Eff}}^{\,k}(X/G)\ \subset\ \textrm{Nef}^{\,k}(X/G).
\]
Now assume $\textrm{Nef}^{\,k}(X)^G\,=\,\overline{\textrm{Eff}}^{\,k}(X)^G$. Then all solid arrows in the diagram above are bijections; thus, the dotted arrow is also a bijection. So, the desired equality follows.

The reverse implication follows from a similar argument, using a diagram analogous to the one above in which the roles of pullback and pushforward have been interchanged.
\end{proof}

\begin{corollary}
If $\overline{\emph{Eff}}^{\,k}(X)\ =\ \emph{Nef}^{\,k}(X)$, then $\overline{\emph{Eff}}^{\,k}(X/G)\ =\ \emph{Nef}^{\,k}(X/G)$ for all k.    
\end{corollary}

\begin{remark}
We note that Theorem \ref{lower} implies that both the pullback and pushforward maps
\begin{equation*}
N_{\ast}(X/G)\,\longrightarrow\,N_{\ast}(X)^{G} \quad\textrm{and}\quad N_{\ast}(X)^{G}\,\longrightarrow\, N_{\ast}(X/G)
\end{equation*}
are isomorphisms. Consequently, it suffices to determine the generators of just one of these groups, as
the corresponding generators in the other group may be immediately determined via the knowledge of
the appropriate pullback or pushforward map.
\end{remark}

\section*{Acknowledgements}

We are very grateful to the anonymous referee for suggesting a broader question and providing a supporting 
reference, which led to a substantial improvement of the manuscript. The first and third named authors are 
grateful to Professor D. S. Nagaraj for many helpful discussions.  The second named author is partially 
supported by a J. C. Bose Fellowship (JBR/2023/000003).


\begin{thebibliography}{llllllll}


\bibitem[Da]{Da} N.B. Dang, Degrees of iterates of rational maps on normal projective varieties, \emph{Proc. Lond. Math. Soc.} {\bf 121} (2020), 1268--1310.
We use the result from the arXiv version: arXiv:1701.07760.

\bibitem[Fu]{Fu} W. Fulton, \emph{Intersection theory, Second edition}, Springer-Verlag, Berlin, 1998.


\bibitem[FL]{FL} M. Fulger and B. Lehmann, Positive cones of dual cycle classes, \emph{Algebraic Geometry}
{\bf 4} (2017), 1--28.

\bibitem[Ful]{Ful} M. Fulger, The cones of effective cycles on projective bundles over curves, \emph{Math. 
Zeit.} {\bf 269} (2011), 449--459.

\bibitem[SP]{SP} The Stacks Project Authors, Stacks Project, {\url{https://stacks.math.columbia.edu}} (2018).
\end{thebibliography}
\end{document}